\documentclass[12pt]{amsart}
\usepackage{amsmath,amsfonts,amssymb,amsthm,amstext,pgf,graphicx,hyperref,verbatim,lmodern,textcomp,color,young,tikz}
\usetikzlibrary{decorations}
\usepackage[mathscr]{euscript}
\usetikzlibrary{decorations.markings}
\usetikzlibrary{arrows}
\theoremstyle{plain}
\newtheorem{theorem}{Theorem}[section]

\newtheorem{corollary}{Corollary}

\theoremstyle{definition}

\title{}
\begin{document}
\title [Combinatorial proofs of  Newton-Girard and Chapman-Costas-Santos   identities] {Combinatorial proofs of the Newton-Girard and Chapman-Costas-Santos identities} 
\author[Sajal Kumar Mukherjee]{Sajal Kumar Mukherjee}
\author[Sudip Bera]{Sudip Bera}
\address[Sajal Kumar Mukherjee]{Department of Mathematics, Visva-Bharati, Santiniketan-731235, India.}
\email{shyamal.sajalmukherjee@gmail.com}
\address[Sudip Bera]{Department of Mathematics, Indian Institute of Science, Bangalore 560 012}
\email{sudipbera@iisc.ac.in}
\maketitle

\begin{abstract}
In this paper we give combinatorial proofs of some well known identities and obtain some generalizations. We  give a visual proof of a result of Chapman and Costas-Santos regarding the  determinant of sum of matrices.  Also we find a new identity expressing  permanent of sum of matrices. Besides, we give a graphical interpretation of Newton-Girard identity.
\end{abstract}
\section{Introduction}
In this paper we give combinatorial proofs of some well known identities and obtain some generalizations.  Combinatorial proofs give more insight into  ``why" the result is true rather than ``how" \cite{14, 09,10}.   In Section $2$, we give a graphical interpretation of Newton-Girard identity.  A combinatorial proof of Newton-Girard identity was first given by Doron Zeilberger in \cite{11}. In this paper we have obtained the Newton-Girard identity as a trivial corollary of our more general formulation. Our general formulation gives a relation between weighted sum of closed walk and weighted sum of linear subdigraph of  wighted digraph $\Gamma$. So from our formulation the Newton-Girard identity is exactly same with the relation between weighted sum of closed walk and weighted sum of linear subdigraph of the weighted digraph consisting isolated loops only.  However, to the best of our knowledge  there is no such  general relation between weighted sum of closed walk and weighted sum of  linear subdigraph  of  weighted digraph. In this work we get a new identity expressing the relation between walk and linear subdigraph. 

For a weighted digraph  $\Gamma,$ a \emph{linear subdigraph} $\gamma$  is a  collection of pairwise vertex-disjoint cycles. The number of cycles contained in $\gamma$ is denoted by $c(\gamma)$. The weight of a linear subdigraph $\gamma$ is denoted by $w(\gamma)$.  Denote the set of all linear subdigraph of length $r$ by $L_r$.  Let us denote the sum of the weights of all closed walks of length $r$ by $c_r$. And define 
$\ell_r\triangleq \sum\limits_{\ell\in L_r}(-1)^{c(\ell)}w(\ell).$  Then our  general formulation says that for a weighted digraph $\Gamma$ with $n$ vertices,
\begin{enumerate}
	\item $c_r+c_{r-1}\ell_1+c_{r-2}\ell_2+\cdots+c_{r-n}\ell_n=0, r>n$
	\item $c_r+c_{r-1}\ell_1+c_{r-2}\ell_2+\cdots+r\ell_r=0, 1\leq r\leq n$.
\end{enumerate}
Like  Newton-Girard identity, our general formulation can also be thought as a graph theoretic analogue of the relation between power sum symmetric functions and elementary symmetric functions. So our work  motivates the investigation of possible graph theoretic interpretations for other standard bases of symmetric functions such as complete homogeneous and Schur functions.

Denote by $M_{m,n}$, the set of $m\times n$ matrices over an arbitrary field $F$
and by $M_n$ the set $M_{n,n}$. Consider the $N$-tuple of $n\times n$ matrices $S := (A_1, A_2,\cdots, A_N )$ and $Q\in M_n$. If $N\geq n+1$, then the identity of  Costas-Santos  \cite{12}  says that
\[\sum_{k=0}^N(-1)^k\sum_{\substack{ \sigma\subset[N]\\|\sigma|=k}}\det\left(Q+\sum_{i\in \sigma}A_i\right)=0.\]
Before that in \cite{13}, Chapman proved that 
\[\sum_{k=0}^N(-1)^k\sum_{\substack{ \sigma\subset[N]\\|\sigma|=k}}\det\left(\sum_{i\in \sigma}A_i\right)=0.\]
 However, there is no formula to the best of our knowledge for the permanent of sum of matrices. In this work we get a new formula for the permanent of sum of matrices. 
In section $3$, we give a  visual and combinatorial proof of above identities.
Moreover our method recovers a similar identity for the permanent, which we think to be new.

\section{Graphical interpretation of Newton-Girard identity }
Let $\alpha_1, \alpha_2, \ldots, \alpha_n$ be roots of the polynomial $f(x)=x^n+e_1x^{n-1}+e_2x^{n-2}+\cdots+ e_tx^{n-t}+\cdots+ e_n=0$. Suppose $p_r=\alpha_1^r+\alpha_2^r+\cdots+\alpha_n^r$  $(r=0, 1, \cdots, ) $ so  $p_0=n$. Then  Newton-Girard identity says that
\begin{enumerate}
	\item $p_r+e_1p_{r-1}+e_2p_{r-2}+\cdots+p_1e_{r-1}+re_r=0 (r\leq n)$
	\item $p_r+e_1p_{r-1}+e_2p_{r-2}+\cdots+p_1e_{r-1}+e_np_{r-n}=0 (r> n).$	
\end{enumerate}
In this section we give a graphical interpretation of Newton-Girard identity. In fact the Newton-Girard identity can be deduced as a very special case of our  more general set up.  Before that, let us briefly describe some graph theoretic concepts. See \cite{16} for details. Let $\Gamma$ be a weighted digraph.   A \emph{linear subdigraph} $\gamma$,  of $\Gamma$ is a  collection of pairwise vertex-disjoint cycles. A loop is a cycle of length $1$. So loop around a single vertex is also considered to be a cycle. The weight of a linear subdigraph $\gamma$, written as $w(\gamma)$ is the product of the weights of all its edges. The number of cycles contained in $\gamma$ is denoted by $c(\gamma)$. The length of a linear subdigraph $\gamma$,  denoted by $L(\gamma)$ is the number of edges present in $\gamma$.  Denote the set of all linear subdigraph of length $r$ by $L_r$.  A \emph{walk} in a digraph $\Gamma$ from a vertex $u$ to a vertex $v$ is a sequence of vertices $u=x_0, x_1, \cdots, x_{k-1}, x_k=v $ such that $(x_i, x_{i+1})$ is an edge for $i=0, 1, 2, \cdots, k-1$. The walk is called closed if $u=v$. The length $L(\tilde{w})$ of a walk $\tilde{w}$ is the number of edges present in that walk. The weight $w(\tilde{w})$ of a walk $\tilde{w}$ is the product of all weights of edges present in that walk. Note that when we talk about a closed walk, its initial and final point is automatically specified. Let us denote the sum of the weights of all closed walks of length $r$ by $c_r$. And define
$\ell_r\triangleq \sum\limits_{\ell\in L_r}(-1)^{c(\ell)}w(\ell).$
Now we state the theorem.
\begin{theorem}
For a weighted digraph $\Gamma$ with $n$ vertices, the following identities  hold.
\begin{enumerate}
\item $c_r+c_{r-1}\ell_1+c_{r-2}\ell_2+\cdots+c_{r-n}\ell_n=0, r>n$
\item $c_r+c_{r-1}\ell_1+c_{r-2}\ell_2+\cdots+r\ell_r=0, 1\leq r\leq n$.
\end{enumerate}
\end{theorem}
\begin{proof}
First we prove the case $r>n$. To prove this consider all ordered pair $(c, \gamma)$, where $c$ is a closed walk and $\gamma$ is a linear subdigraph (possibly empty), such that $L(c)+L(\gamma)=r$. Define the weight $W$ of $(c, \gamma)$ to be $W((c, \gamma))=(-1)^{c(\gamma)}w(c)w(\gamma)$. Note that the left hand side of $(1)$ is precisely equal to $\sum\limits_{(c, \gamma)}W((c, \gamma))$, where the summation runs over all ordered pairs $(c, \gamma)$ as described above.

Now the crucial observation is that,  since $r>n,$ either $c$ and $\gamma$ share a common vertex or $c$ is not a ``simple" closed walk (here simple means the graph structure of the closed walk is a directed cycle). Now take a particular pair $(c, \gamma)$ satisfying the above conditions. Suppose that $x$ is the initial and terminal vertex of $c$. Start moving from $x$ along $c$.  There are two possibilities: either, first we meet a vertex $y$ which is a vertex of $\gamma$ or, we complete a closed directed cycle $\acute{c}$ which is a subwalk of $c$ and during this journey from $x$ up to the completion of $\acute{c}$ we have not met any vertex of $\gamma$. Now if the first case holds, we form a new ordered pair $(\tilde{c}, \tilde{\gamma})$, where $\tilde{c}=\widehat{xy}|_{c}\bigodot \gamma_{y}\bigodot \widehat{yx}|_{c}$ and $\tilde{\gamma}=\gamma\setminus\{\gamma_y\}$, where $\widehat{xy}|_{c}$ is the walk from $x$ to $y$ along $c$ and $\gamma_y$ is the directed cycle of $\gamma$ containing the vertex $y$. Note that $W((\tilde{c}, \tilde{\gamma}))=-W((c, \gamma))$. Now if the second case holds, then form a new ordered pair $(\tilde{\tilde{c}}, \tilde{\tilde{\gamma}})$, where $\tilde{\tilde{c}}$ is formed by removing the directed cycle $\acute{c}$ from $c$ and $\tilde{\tilde{\gamma}}$ is $\gamma\cup \acute{c}$. Note also that $W((\tilde{\tilde{c}}, \tilde{\tilde{\gamma}}))=-W((c, \gamma))$. It is easy to see that,  this is in fact an involution and it is sign reversing by the above observation. This completes the proof.

Now we prove the case $r\leq n$. Let $A=\{ (c, \gamma):c$ is a closed walk of length $\geq 1$ and $\gamma$ is linear subdigraph (possibly empty) and $L(c)+L(\gamma)=r\}$. Consider the following sum $S=\sum\limits_{(c, \gamma)\in A}W((c, \gamma))+r\ell_r$. Note that  the left hand side of $(2)$ is precisely equal to $S$.

Consider the subset of $A$ consisting of ordered pair $(c, \gamma)$ satisfying the conditions: either $c\cap \gamma\neq \phi$ or $c$ is not a simple closed walk. Call this subcollection \emph{BAD}. So the \emph{GOOD} members of $A$ are the ordered pairs $(c, \gamma)$ satisfying $c\cap \gamma=\phi$ and $c$ is a simple closed walk. Now observe that,  the weights of the BAD members cancel among themselves just like the previous case (case, $r>n$). Now let us see,  how a GOOD member looks like. As a directed graph it is just a disjoint collection of distinct cycles with vertex set $\{ v_1, v_2, \cdots, v_r\}$ i. e. it is a linear subdigraph $\dot{\gamma}$ with vertex set $\{v_1, v_2, \cdots, v_r\}$. Now for this fixed  linear subdigraph $\dot{\gamma}$ with vertex set $\{v_1, v_2, \cdots, v_r\}$, we claim that there are precisely $r$ GOOD members $(c, \gamma)$. For the proof, take any vertex say $v_i$ from $\dot{\gamma}$. Consider the cycle $c$ in $\dot{\gamma}$ containing the vertex $v_i$. Let $\gamma_1=\dot{\gamma}\setminus \{ c\}$. Now the cycle $c$ can be thought of as a closed walk $c_{v_i}$ starting and ending at the vertex $v_i$. So we get a GOOD member $(c_{v_i}, \gamma_1)$. Since $v_i$ is arbitrary the claim follows.

The main observation is that the sum of the weights of all the good members, found in this way from $\dot{\gamma}$ is $r(-1)^{c(\dot{\gamma})-1}w(\dot{\gamma})$. This cancels with the term $r(-1)^{c(\dot{\gamma})}w(\dot{\gamma})$ in the equation $S=\sum\limits_{(c, \gamma)\in A}W((c, \gamma))+r\ell_r$.
\end{proof}

\begin{corollary}[Newton-Girard identity]
 Let $\alpha_1, \alpha_2, \ldots, \alpha_n$ be roots of the polynomial $f(x)=x^n+e_1x^{n-1}+e_2x^{n-2}+\cdots+ e_tx^{n-t}+\cdots+ e_n=0$. Suppose $p_r=\alpha_1^r+\alpha_2^r+\cdots+\alpha_n^r$  $(r=0, 1, \cdots, )$ so $p_0=n$. Then Newton-Girard identity says that
\begin{enumerate}
\item $p_r+e_1p_{r-1}+e_2p_{r-2}+\cdots+p_1e_{r-1}+re_r=0 (r\leq n)$
\item $p_r+e_1p_{r-1}+e_2p_{r-2}+\cdots+p_1e_{r-1}+e_np_{r-n}=0 (r> n)$	
\end{enumerate}
\end{corollary}
	

\begin{proof}
Just apply the above theorem on the  weighted digraph consisting of $n$ disjoint loops with weight $\alpha_1, \alpha_2, \ldots, \alpha_n$ respectively, we get the Newton-Girard identity.
\end{proof}

\section{Combinatorial proof of Chapman-Costas-Santos identity}
In this section, we prove  interesting identities about the determinant and permanent of sum of matrices. Before that we have to develop some necessary background. For details we refer  \cite{15}.  Let $G$ be a weighted, acyclic digraph.  We call  $A=\{ A_1, A_2, \cdots, A_n\}$ to be the initial set of vertices  and $B=\{B_1, B_2, \cdots, B_n\}$ to be the terminal set of vertices (not necessarily disjoint) of the graph $G$. To $A$ and $B$, associate the \emph{path matrix} $M=(m_{ij})$, where $m_{ij}=\sum\limits_{P:A_i\rightarrow B_j}w(P),$ [$w(P)$ is the product of weights of all edges involved in the path $P$]. A \emph{path system} $\mathcal{P}$ from $A$ to $B$ consists of a permutation $\sigma$ and $n$ paths $P_i: A_i\rightarrow B_{\sigma(i)}$, with $\text{sgn} (\mathcal{P})=\text{sgn} (\sigma)$. The \emph{weight} of $\mathcal{P}$ is $w(\mathcal{P})=\prod_{i=1}^nw(P_i)$.  Now it easy to see that ${\det}(M)=\sum\limits_{\mathcal{P}}\text{sgn}(\mathcal{P})w(\mathcal{P})$ 
and $\operatorname{per}(M)=\sum\limits_{\mathcal{P}}w(\mathcal{P})$.

For the $N$-tuple of $n\times n$ matrices $S := (A_1, A_2,\cdots, A_N )$ and an $n\times n$ matrix $Q=(q_{i,j})$ with  $N\geq n+1$,  Costas-Santos \cite{12} proved
 \[\sum_{k=0}^N(-1)^k\sum_{\substack{ \sigma\subset[N]\\|\sigma|=k}}\det\left(Q+\sum_{i\in \sigma}A_i\right)=0\] and before that Chapman \cite{13} proved
  \[\sum_{k=0}^N(-1)^k\sum_{\substack{ \sigma\subset[N]\\|\sigma|=k}}\det\left(\sum_{i\in \sigma}A_i\right)=0\] 
  by using algebraic manipulations. In this section we give a purely combinatorial proof of his identity and the same identity for permanent.

  \begin{figure}[ht!]
  	\tiny
  	\tikzstyle{ver}=[]
  	\tikzstyle{vert}=[circle, draw, fill=black!100, inner sep=0pt, minimum width=4pt]
  	\tikzstyle{vertex}=[circle, draw, fill=black!00, inner sep=0pt, minimum width=4pt]
  	\tikzstyle{edge} = [draw,thick,-]
  	\tikzstyle{node_style} = [circle,draw=blue,fill=blue!20!,font=\sffamily\Large\bfseries]
  	\centering
  	\tikzset{->,>=stealth', auto,node distance=1cm,
  		thick,main node/.style={circle,draw,font=\sffamily\Large\bfseries}}
  	\tikzset{->-/.style={decoration={
  				markings,
  				mark=at position #1 with {\arrow{>}}},postaction={decorate}}}
  	
  	\begin{tikzpicture}[scale=1]
  	\tikzstyle{edge_style} = [draw=black, line width=2mm, ]
  	\tikzstyle{node_style} = [draw=blue,fill=blue!00!,font=\sffamily\Large\bfseries]
  	\node (B1) at (-.3,0)   {$\bf{B_1}$};
  	\node (B2) at (-.3,1.5)  {$\bf{B_3}$};
  	\node (B3) at (1.9,0)   {$\bf{B_2}$};
  	\node (B4) at (1.9,1.5)   {$\bf{B_4}$};
  	\node (A1) at (-6.4,0)   {$\bf{A_1}$};
  	\node (A2) at (-4,0)   {$\bf{A_2}$};
  	\node (A3) at (-6.4,1.5)  {$\bf{A_3}$};
  	\node (A4) at (-4,1.4)  {$\bf{A_4}$};
  	\node (C1) at (4.7,.1)   {$\bf{C_1}$};
  	\node (C2) at (6.8,0)  {$\bf{C_2}$};
  	\node (C3) at (4.6,1.4) {$\bf{C_3}$};
  	\node (C4) at (6.8,1.5) {$\bf{C_4}$};
  	\node (X1) at (0,4.2) {$\bf{X_1}$};
  	\node (X2) at (1,4.2) {$\bf{X_2}$};
  	\node (Y1) at (0,-4.2) {$\bf{Y_1}$};
  	\node (Y2) at (1,-4.2) {$\bf{Y_2}$};
  	\node (A) at (-6,-2)  {$\textbf{A}$};
  	\node (B) at (.5,-2)  {$\textbf{B}$};
  	\node (C) at (6,-2)  {$\textbf{C}$};
  	\node (1) at (-.4,1)  {$\bf{b_{11}}$};
  	\node (1) at (1.9,1) {$\bf{b_{22}}$};
  	\node (1) at (.72,.1)  {$\bf{b_{21}}$};
  	\node (1) at (.6,1.35)  {$\bf{b_{12}}$};
  	
  	\draw[->, line width=.2 mm] (0,1.5) -- (-.01,.02);
  	\draw[->, line width=.2 mm] (0,1.5)--(1.5,0);
  	\draw[->, line width=.2 mm] (1.5,1.5)--(0,-.01);
  	\draw[->, line width=.2 mm] (1.5,1.5)--(1.53,0);
  	
  	\draw  (-1,-1) rectangle (2.4,2);
  	
  	\draw[->, line width=.2 mm] (-4.5,1.5) -- (-4.47,0);
  	\draw[->, line width=.2 mm] (-4.5,1.5)--(-6,0);
  	\draw[->, line width=.2 mm] (-6,1.5)--(-6.03,0);
  	\draw[->, line width=.2 mm] (-6,1.5)--(-4.5,0);
  	\node (1) at (-6.4,1)  {$\bf{a_{11}}$};
  	\node (1) at (-4,1)  {$\bf{a_{22}}$};
  	\node (1) at (-5.3,.2)  {$\bf{a_{21}}$};
  	\node (1) at (-5.2,1.3)  {$\bf{a_{12}}$};
  	\draw  (-7,-1) rectangle (-3.5,2);

  	\draw[->, line width=.2 mm] (5,1.5) -- (4.97,0);
  	\draw[->, line width=.2 mm] (5,1.5)--(6.5,0);
  	\draw[->, line width=.2 mm] (6.5,1.5)--(6.53,0);
  	\draw[->, line width=.2 mm] (6.5,1.5)--(5,0);
  	\node (1) at (4.5,1)  {$\bf{c_{11}}$};
  	\node (1) at (6.9,1)  {$\bf{c_{22}}$};
  	\node (1) at (5.6,.2)  {$\bf{c_{21}}$};
  	\node (1) at (5.6,1.3)  {$\bf{c_{12}}$};
  	\draw  (4,-1) rectangle (7.3,2);

  	\draw (.5,4) ellipse (3cm and 1cm);
  	\draw (.5,-4) ellipse (3cm and 1cm);
  	
  	\draw[->, line width=.2 mm] (0,4) -- (-6,1.5);
  	\draw[->, line width=.2 mm] (1,4) -- (-4.5,1.5);
  	\draw[->, line width=.2 mm] (0,4) -- (0,1.5);
  	\draw[->, line width=.2 mm] (1,4) -- (1.5,1.5);
  	\draw[->, line width=.2 mm] (0,4) -- (5,1.5);
  	\draw[->, line width=.2 mm] (1,4) -- (6.5,1.5);
  	\draw[->, line width=.2 mm] (-6,0) -- (0,-4);
  	\draw[->, line width=.2 mm] (-4.5,0) -- (1,-4);
  	\draw[->, line width=.2 mm] (1.5,0) -- (1,-4);
  	\draw[->, line width=.2 mm] (0,0) -- (0,-4);
  	\draw[->, line width=.2 mm] (5,0) -- (0.01,-4);
  	\draw[->, line width=.2 mm] (6.5,0) -- (1.02,-4);

  	
  	
  	
  	

  	\node (1) at (-3,3)  {$\bf{1}$};
  	\node (1) at (-3,2.4)  {$\bf{1}$};
  	\node (1) at (-3,-2.2)  {$\bf{1}$};
  	\node (1) at (-3,-1.4) {$\bf{1}$};
  	\node (1) at (-.4,2.2)  {$\bf{1}$};
  	\node (1) at (-.4,-2)  {$\bf{1}$};
  	\node (1) at (1.5,2.2) {$\bf{1}$};
  	\node (1) at (1.5,-2.2)  {$\bf{1}$};
  	\node (1) at (3.6,3.1)  {$\bf{1}$};
  	\node (1) at (3,2.2)  {$\bf{1}$};
  	\node (1) at (3,-2.2)  {$\bf{1}$};
  	\node (1) at (3,-1.4)  {$\bf{1}$};
  	
  	\end{tikzpicture}
  	\caption{First box, second box and third box are denoted by $\textbf{A, B, C}$ respectively.}
  	\label{fig:f2}	
  \end{figure}
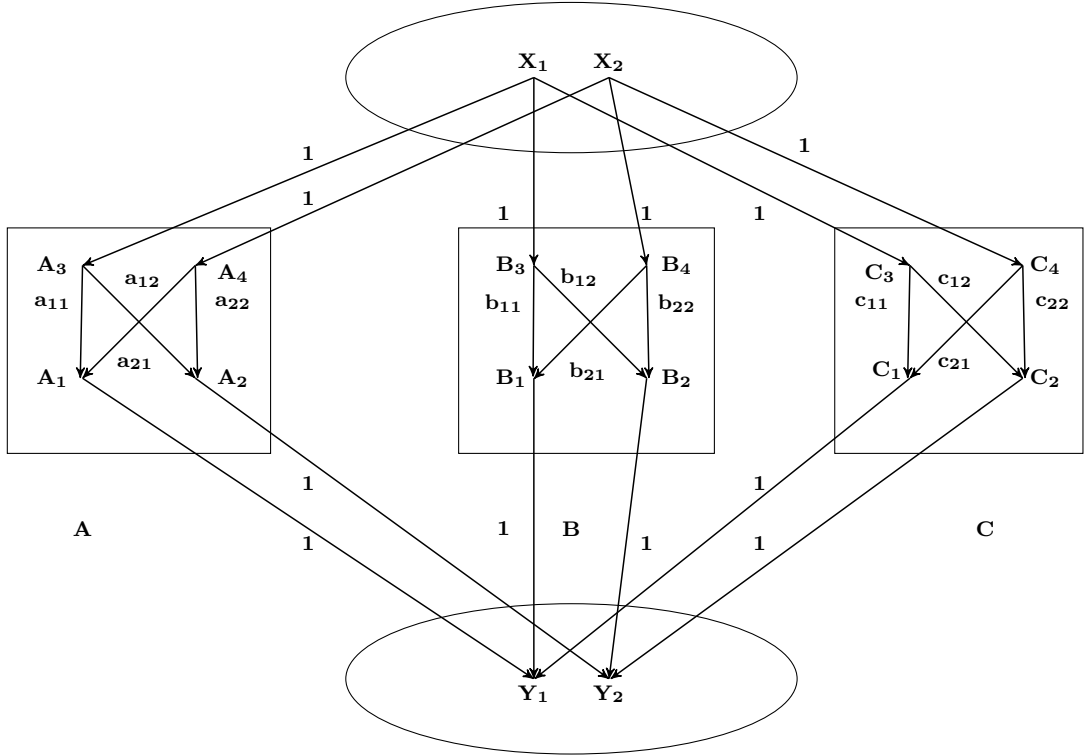

\begin{theorem}
Consider the $N$-tuple of $n$-by-$n$ matrices $S := (A_1, A_2,\cdots, A_N )$. If $N\geq n+1$, then the following relations hold.
\begin{enumerate}
\item $\sum\limits_{k=0}^N(-1)^k\sum\limits_{\substack{ \sigma\subset[N]\\|\sigma|=k}}\det\left(\sum\limits_{i\in \sigma}A_i\right)=0.$
\item $\sum\limits_{k=0}^N(-1)^k\sum\limits_{\substack{ \sigma\subset[N]\\|\sigma|=k}}\operatorname{per}\left(\sum\limits_{i\in \sigma}A_i\right)=0.$
\end{enumerate}
\end{theorem}
\begin{proof}
For the sake of simplicity we take $N=3$, $n=2$ and we prove these identities for three matrices $A=(a_{i,j})_{2\times 2}, B=(b_{i,j})_{2\times 2}, C=(c_{i, j})_{2\times 2}$ . The proofs of these  two identities follow from the digraph in Figure \ref{fig:f2} and  the \emph{Principle of Inclusion and Exclusion (PIE)}.

Look at the boxes labeled by $\textbf{A, B, C}$  in Figure \ref{fig:f2}. From this figure,   clearly the path matrix (paths from vertex set $X=\{x_1, x_2\}$ to the vertex set $Y=\{y_1, y_2\}$) of the  graph is the matrix $A+B+C$. Let $\mathcal{P}_{ABC}=\{ \mathcal{P},$ a path system from $X$ to $Y$ such that each box contains atleast $2$ intermediate vertices of some paths in $\mathcal{P}$$ \}$. Then $\sum\limits_{\mathcal{P\in \mathcal{P}_{ABC}}} \text{sgn}(\mathcal{P})w(\mathcal{P})=0$, as $\mathcal{P}_{ABC}=\varnothing$ (because each path system contains exactly two paths which can not pass through all three boxes simultaneously). Now we compute this empty sum by \emph{PIE}.  The signed sum of  weights of all possible path systems is $\det(A+B+C)$. The signed sum of  weights of all possible path systems,  whose underlying paths do not pass through box $\textbf{C}$ is $\det(A+B)$. Similarly, the signed sum of  weights of all possible path systems,  whose underlying paths do not pass through box $\textbf{B}$ is $\det(A+C)$ and the signed sum of  weights of all possible path systems,  whose underlying paths do not pass through box $\textbf{A}$ is $\det(B+C)$. Again the signed sum of  weights of all possible path systems, whose underlying paths pass through neither $B$ nor $C$ is $\det(A)$. Proceeding  this way and using the \emph{PIE} we get,
\begin{align*}
&\det(A+B+C)-\det(A+B)-\det(A+C)-\det(B+C)+\det(A)+\det(B)+\det(C)\\
=&\sum_{\mathcal{P}\in \mathcal{P_{ABC}}}\text{sgn}(\mathcal{P})w(\mathcal{P})=0.
\end{align*}
Similarly,
\[\operatorname{per}(A+B+C)-\operatorname{per}(A+B)-\operatorname{per}(A+C)-\operatorname{per}(B+C)+\operatorname{per}(A)+\operatorname{per}(B)+\operatorname{per}(C)=0.\]

Now if we add one more directed edge from each vertex $X_i (i=1, 2)$ to each vertex $Y_j (j=1, 2)$ in the directed graph in Figure \ref{fig:f2} with weight $q_{i,j}$ and applying above argument we can prove Costas-Santos's identity  
\[\sum_{k=0}^N(-1)^k\sum_{\substack{ \sigma\subset[N]\\|\sigma|=k}}\det\left(Q+\sum_{i\in \sigma}A_i\right)=0.\]
Similarly we can prove the new identity 
\[\sum_{k=0}^N(-1)^k\sum_{\substack{ \sigma\subset[N]\\|\sigma|=k}}\operatorname{per}\left(Q+\sum_{i\in \sigma}A_i\right)=0.\]
\end{proof}

\subsection*{Acknowledgement} We would like to thank Prof. Arvind Ayyer and Prof. Darij Grinberg for valuable  suggestions in the preparation of this paper. The second author was supported by Department of Science and Technology grant EMR/2016/006624 and partly supported by  UGC Centre for Advanced Studies.

\bibliographystyle{amsplain}
\bibliography{gen-inv-lcp}

\end{document}